\documentclass[a4paper,oneside]{amsart}
\usepackage[utf8]{inputenc}
\usepackage[margin=1in, top=1in, footskip=.25in]{geometry}
\usepackage{amsfonts}
\usepackage{amsmath}
\usepackage{amssymb}  
\usepackage{amsthm}
\usepackage{mathtools}
\usepackage{graphicx}
\usepackage{enumitem}
\usepackage{mathrsfs}
\usepackage{color}
\usepackage{bm}
\usepackage{float}
\usepackage{subcaption}
\usepackage{braket}
\usepackage{comment}
\usepackage{tikz-cd}

\providecommand{\keywords}[1]{\noindent \textbf{\textit{Keywords:}} #1}
\newtheorem{theorem}{Theorem}[section]
\newtheorem{proposition}[theorem]{Proposition}
\newtheorem{lemma}[theorem]{Lemma}
\newtheorem{corollary}[theorem]{Corollary}

\theoremstyle{definition}
\newtheorem{definition}[theorem]{Definition}
 
\newtheorem{remark}[theorem]{Remark}

\theoremstyle{remark}
\newtheorem*{notation}{Notation}

\newcounter{claimcount}
\newcommand{\framedH}[1]{\mathcal{H}^{#1}} 
\newcommand{\framedHn}{\mathcal{H}^n} 
\DeclareMathOperator{\PSL}{PSL} 
\newcommand{\bH}{\mathbb H}
\DeclareMathOperator{\R}{\mathbb R}		
\DeclareMathOperator{\bZ}{\mathbb Z}		
\DeclareMathOperator{\CC}{\mathsf{CC}} 
\DeclareMathOperator{\Sub}{Sub}			
\DeclareMathOperator{\inj}{inj}			
\DeclareMathOperator{\diam}{diam}			

\newcommand{\quotient}[2]{\left.\raisebox{.1em}{$#1\!$}\middle/\raisebox{-.1em}{$#2$}\right.}
\newcommand{\bquotient}[2]{\left.\raisebox{-.1em}{$#1$}\middle\backslash\raisebox{.1em}{$#2$}\right.}

\begin{document}

\begin{abstract}
	In the space  $\mathcal{H}^2$ of hyperbolic surfaces decorated with a base unit vector, the topology induced by the Gromov-Hausdorff convergence  coincides with the Chabauty topology on the space of discrete torsion-free subgroups of $\rm{PSL}_2(\mathbb{R})$. Using paths constructed from changing the Fenchel-Nielsen coordinates and shrinking simple closed curves to cusps, we demonstrate path-connectivity of  $\mathcal{H}^2$ and some of its subspaces. 
\end{abstract}

\title{The Space of Vectored Hyperbolic Surfaces is Path-Connected}
\author{Sangsan Warakkagun}

\address{Beijing Institute of Mathematical Sciences and Applications, Huairou District, Beijing, China}
\curraddr{Department of Mathematics, Faculty of Science, Khon Kaen University, Khon Kaen, 40002, Thailand}

\email{sangwa@kku.ac.th}

\maketitle

%\tableofcontents
	
	\section{Introduction}
	Let the space of all vectored hyperbolic surfaces
	\[
	\framedH2 := \quotient{\Set{ (X,p,\textbf{v}) :\begin{array}{c}
				X \text{ an oriented hyperbolic surface,} \\ 
				p \in X, \textbf{v} \in T^1_p X 
	\end{array}}\!}{\!\begin{array}{c}
	\text{orientation- and base } \\ \text{vector-preserving isometry}
\end{array}}
	\]
	be endowed with the Gromov-Hausdorff topology (see Definition \ref{defn: GH}). All underlying surfaces in $\framedH2$ are connected, oriented, complete, and have no boundary. These surfaces can be of finite of infinite topological type. Informally, two vectored hyperbolic surfaces are near in $\framedH2$ if there is a correspondence between points in large compact balls centered at the basepoints that approximates an isometry and also respects the base vectors. 
	
	An alternate perspective of $\framedH2$ comes from the Chabauty topology. Letting $G= \PSL_2(\R)$, we define  $\Sub(G)$ to be the space of closed subgroups equipped with the Chabauty convergence, see Definition \ref{defn: Chabauty}. Then, 
	\[ \framedH2 \approx \Sub_{DT}(G):= \{ \text{discrete torsion-free subgroups of } G\} \subset \Sub(G). \]
	More precisely, if we fix a unit vector $(p_{\text{base}}, \textbf{v}_{base})$ in the unit tanget bundle $T^1\bH^2$, a homeomorphism is given by  
	\[
	\Phi: \Sub_{DT}(G)  \to \ \framedH2, \qquad
	\Gamma \mapsto \ (\bquotient{\Gamma}{\bH^2},\pi_{\Gamma}(p_{\text{base}}), d\pi_{\Gamma}(\textbf{v}_{base})),
	\]
	where $\pi_\Gamma: \bH^2 \to \bquotient{\Gamma}{\bH^2}$ is the projection and $d\pi_\Gamma: T^1\bH^2 \to \bquotient{\Gamma}{T^1\bH^2}$ is its derivative, see details in \cite[Chapter E]{Benedetti1992}.	

The space $\framedH2$ is the 2-dimensional instance of the space $\framedH{n}$ of framed hyperbolic $n$-manifolds, in which context the topology is known as the \textit{geometric topology}. Perhaps most crucially, $\framedH{n}$ is the setting in the classical works of J{\o}rgensen and Thurston to study the volume spectrum of hyperbolic $n$-manifolds \cite[Chapters 5 \& 6]{Thurston2022}. For $n=2$, the range of the volume function for hyperbolic surfaces is just $2\pi \mathbb{N} \cup \{ \infty\}$ by the Gauss-Bonnet Theorem. 

There have already been some results regarding the topology of $\framedH2$ itself or a related space. With respect to the Chabauty topology, Baik and Clavier proved in \cite{Baik2013} that the closure of the space of one-generator subgroups of $G$ is simply connected. In \cite{Biringer2021}, Biringer, Lazarovich, and Leitner showed that the closure of the subspace $\Sub_S(G)$ of lattices $\Gamma < G$ for which $\bquotient{\Gamma}{\bH^2} \approx S$ is simply-connected when $S$ is a finite-type hyperbolic 2-orbifold with sufficiently large complexity. The main result of this paper is the following:

\begin{theorem}\label{globally connected}
$\framedH2$ is path-connected.
\end{theorem}
In the proof of Theorem \ref{globally connected}, we construct a path from any vectored hyperbolic surface $(X, p, \textbf{v})$ to $(\bH^2, p_{base}, \textbf{v}_{base})$ by connecting both to a vectored thrice punctured sphere. The subtlety of  choosing the base vector rather than just a basepoint along a path is handled in Section \ref{paths}. Our proof strategy also allows us to derive path-connectivity for certain subspaces of $\framedH2$, which we discuss in Section \ref{proof}.

	\section{Preliminaries}

	\subsection{Geodesic pants decomposition}	
	A {\it hyperbolic surface} is a surface equipped with a Riemannian metric of constant negative curvature $-1$ (called a \textit{hyperbolic metric}). A hyperbolic surface $X$ is {\it complete} if it is complete as a metric space, and \textit{geodesically complete} if it is complete and has no boundary. The \textit{convex core} of $X$, denoted by $\CC(X)$, is the smallest closed convex subsurface of $X$ with (possibly empty) geodesic boundary that is homotopy equivalent to $X$. We say that $X$ is of \textit{finite topological type} if $\pi_1(X)$ is finitely generated; it is of \textit{infinite topological type} otherwise. 
	
	By a \textit{generalized pair of pants}, we mean a complete hyperbolic surface diffeomorphic to a sphere with three open disks or points removed. We require that it includes each of its boundary components that is a simple closed geodesic. Below, we record a version of a geodesic pants decomposition that also addresses hyperbolic surfaces of infinite topological type. A proof can be found in \cite{Alvarez2004}.
	
	\begin{theorem}[Geodesic pants decomposition]\label{geo-pants}
	Let $X$ be a complete hyperbolic surface whose $\pi_1(X)$ is non-abelian  and whose boundary is a (possibly empty) union of disjoint simple closed geodesics. Let $\mathcal{L}$ be the set of boundary components of the convex core $\CC(X)$ which are open infinite geodesics (if there are any). 
	
	Then, there is a collection $\mathcal{P}$ of disjoint simple closed geodesics  in $\CC(X)$ so that the closure of each component of $\CC(X) \setminus (\mathcal{L} \cup \mathcal{P})$ is a generalized pair of pants. We call $\mathcal{P}$ a geodesic pants decomposition of $X$. For our purpose, we will  include in $\mathcal{P}$ the curves representing the cusps in $X$ and the boundary components of $\CC(X)$ which are simple closed geodesics 
	\end{theorem}
	
In the case that $X$ is a complete hyperbolic surface with no boundary and $\pi_1(X)$ is abelian, then either $\pi_1(X)$ is trivial or $\pi_1(X) \cong \bZ$. That is, $X$ is isometric to one of the following surfaces:
\begin{enumerate}[label = \arabic*), itemsep=0.5ex]
	\item the hyperbolic plane: $\bH^2 = \{ z \in \mathbb{C}: \text{Im} z > 0\}$,
	\item a \textit{one-cusp cylinder}: $ \langle z \mapsto z + c \rangle \backslash\bH^2$, for some $c\in \R$, or
	\item an \textit{infinite funnel}: $ \langle z \mapsto e^{\ell}z \rangle \backslash \bH^2$ for some $\ell > 0$.
\end{enumerate}

	\subsection{The topology on $\framedH2$}\label{define-H2}
	In this section, we define the space $\framedH2$ and the Gromov-Hausdorff topology. We follow the conventions set in \cite[Chapter I.3]{Canary2006}. All hyperbolic surfaces in this subsection are connected, oriented, and geodesically complete.

	\begin{definition}
	A {\it vectored hyperbolic surface} is a triple $(X,p,\textbf{v})$ of a hyperbolic surface $X$, a point $p \in X$ (called a \textit{basepoint}), and a unit vector $\textbf{v} \in T^1_p(X)$ (called a \textit{base vector}).	We usually write the choice of base vector as $(p,\textbf{v}) \in T^1X$.
	\end{definition}

	Let $\framedH2$ be the set of vectored hyperbolic surfaces, considered up to vector-preserving isometries:
	\[ \framedH2 = \{ \text{vectored oriented hyperbolic surfaces }  (X,p,\textbf{v}) \}/\sim, \]
	where $(X,p,\textbf{v}) \sim (X',p', \textbf{v}')$ if there is an orientation-preserving isometry $f: X \to X'$ such that $df(p,\textbf{v}) =(p',\textbf{v}')$. Here, $df$ denotes the differential of $f$. 
	
	\begin{notation}
	In $\framedH2$, we often suppress the distinction between a vectored surface and the isometry class that it represents unless ambiguities arise.
	\end{notation}

\begin{definition}
	For $\epsilon > 0$ and $R >0$, two vectored hyperbolic surfaces $(X, p, {\bf v})$ and $(Y, q, {\bf w})$ in $\framedH2$ are {\it $(\epsilon,R)$-related} if there are compact subsurfaces $X_1 \subset X$ and $Y_1 \subset Y$, where $B_X(p, R) \subset X_1$ and $B_Y(q, R) \subset Y_1$, together with a relation $\mathfrak{R} \subset X_1 \times Y_1$  satisfying the following conditions 
	\begin{enumerate}[label = \arabic*), itemsep=0.5ex]
		\item $(p, q) \in \mathfrak{R}$;
		\item for every $x \in X_1$, there is a point $y \in Y_1$ such that $(x,y) \in\mathfrak{R}$;
		\item for every $y \in Y_1$, there is a point $x \in X_1$ such that $(x,y) \in\mathfrak{R}$;
		\item if $(x_1, y_1), (x_2,y_2) \in \mathfrak{R}$, then  $|d_Y(y_1,y_2) - d_{X}(x_1, x_2)| < \epsilon$;
		\item if $\textbf{v}' \in T^1_pX$ and $\textbf{w}' \in T^1_qY$ are obtained by rotating $\textbf{v}$ and $\textbf{w}$ counterclockwise by $\pi/2$, then 
		\[ r_1^2+r_2^2 \leq R^2 \Longrightarrow (\exp(r_1 \textbf{v} + r_2\textbf{v}'),  \exp(r_1 \textbf{w} + r_2\textbf{w}')) \in \mathfrak{R}, \]
 where $\exp$ is the exponential map.
	\end{enumerate}
\end{definition}

	We are now ready to define the topology on $\framedH2$.
	\begin{definition}\label{defn: GH}
	The \textit{Gromov-Hausdorff topology} on $\framedH2$ is generated by neighborhoods of the form
	\begin{equation*}
		\mathfrak{N}(X, p, \textbf{v}, \epsilon, R) = \{(Y,q, \textbf{w}) \in \framedH2: (X,p, \textbf{v}) \text{ and } (Y,q, \textbf{w}) \text{ are } (\epsilon,R)\text{-related}\}
	\end{equation*}
	where $(X,p,v)$ ranges over vectored hyperbolic surfaces in $\framedH2$ and $\epsilon, R >0$.
	\end{definition}
	
	When equipped with this topology, $\framedH2$ is Hausdorff,  see \cite[Chapter I.3]{Canary2006}. Thus, there is a well-defined notion of convergence. Indeed, a sequence $\{(X_i, p_i, \textbf{v}_i)\}_{i=1}^\infty$ converges to $(X_{\infty}, p_{\infty}, \textbf{v}_{\infty})$ in $\framedH2$ if and only if there exist sequences $\epsilon_i \to 0$ and $R_i \to \infty$ such that $(X_i, p_i, \textbf{v}_i)$ and $(X_{\infty}, p_{\infty}, \textbf{v}_{\infty})$ are $(\epsilon_i, R_i)$-related.

	\subsubsection{The Chabauty topology and convergence in $\framedH2$} \label{sec: chab}
We will make use of equivalent conditions for convergence in $\framedH2$. These criteria are established via the Chabauty topology, which we briefly recall below. 
	
	\begin{definition}\label{chabauty}\label{defn: Chabauty}
	Let $G$ be a Lie group and let $\Sub(G)$ be the set of closed subgroups of $G$.  We say that a sequence $\{H_i\}_{i=1}^\infty$ in $\Sub(G)$ \textit{Chabauty converges to} $H \in \Sub(G)$ if and only if
	\begin{itemize}[itemsep=0.5ex]
		\item each $h \in H$ is the limit of some sequence $\{h_i\}_{i=1}^\infty$, where $h_i \in H_i$ for every $i$, and
		\item if $\{h_{i_n}\}_{n = 1}^\infty \subset \bigcup\limits_{i=1}^\infty H_i$  is a convergent sequence with $h_{i_n} \in H_{i_n}$ and $h_{i_n} \to h$ in $G$, then $h \in H$.
	\end{itemize}
	\end{definition}

The Chabauty topology as characterized by this convergence makes $\Sub(G)$ a compact metrizable space. An explicit metric is described in \cite{Biringer2017}, for instance. The Chabauty topology gives interesting compactifications for many objects in geometry and group theory via the theory of invariant random subgroups; the interested reader may find references in  \cite{Gelander2015} useful.
	 
	 Now taking $G = \PSL_2(\R)$, the group of orientation-preserving isometries of $\bH^2$, we let $\Sub_{DT}(G) \subset \Sub(G)$ be the subspace of discrete torsion-free subgroups of $G$. With a fixed $(p_{base}, \textbf{v}_{base}) \in T^1\bH^2$, we have a well-defined homeomorphism 
	\begin{equation}\label{chabauty-corr}
		\begin{split}
			\Phi: \Sub_{DT}(G) &\to \framedH2 \\
			\Gamma &\mapsto (\bquotient{\Gamma}{\bH^2}, \pi_\Gamma(p_{base}), d\pi_\Gamma(\textbf{v}_{base})),
		\end{split}
	\end{equation}
	where $\pi_\Gamma: \bH^n \to \bquotient{\Gamma}{\bH^2}$ is the quotient map, see details in \cite{Benedetti1992}.  In fact, the map $\Phi$ in (\ref{chabauty-corr})  naturally extends to a homeomorphism between the Chabauty space of discrete subgroups $\Sub(G)$ and the space of vectored hyperbolic 2-orbifolds, as in  \cite{Biringer2021}. We will focus only on hyperbolic surfaces in this paper. 
	
	Besides $G= \PSL_2(\R) $, the topology of the Chabauty spaces of closed subgroups for some other groups have been either fully or partially studied, including $\R^n$ \cite{Kloeckner2009}, the 3-dimensional Heisenberg group \cite{Bridson2009}, locally compact abelian groups \cite{Cornulier2011}, and $\text{SL}_3(\R)$ \cite{Leitner2016, Lazarovich2021}. 
	
	To summarize, we list some equivalent convergence criteria in $\framedH2$. These are well-known, and we give a reference for where each item's equivalence with the Chabauty convergence is proved accordingly.
	
	\begin{proposition}[Convergence criteria in $\framedH2$]\label{conv-crit}
	Let $\{(X_i,p_i,\textbf{v}_i)\}_{i=1}^\infty$ be a sequence of vectored hyperbolic surfaces in $\framedH2$ and $(X_\infty, p_\infty, \textbf{v}_\infty) \in \framedH2$.  Then, all of the following criteria are equivalent to the convergence $(X_i, p_i, \textbf{v}_i) \to (X_{\infty}, p_{\infty}, \textbf{v}_\infty)$ in $\framedH2$:
	\begin{enumerate}
		
		\item (via the Chabauty topology) If $\Phi$ is the map from (\ref{chabauty-corr}) above, then $\Phi^{-1}(X_i, p_i, \textbf{v}_i)$ Chabauty converges to $\Phi^{-1}(X_\infty, p_\infty, \textbf{v}_\infty)$ as subgroups of $\PSL_2(\R)$.
				
		\item (via $(\epsilon,R)$-relations, \cite[Theorem I.3.2.9]{Canary2006}) There are sequences $\epsilon_i \to 0$ and $R_i \to \infty$  such that $(X_i, p_i, \textbf{v}_i)$ and $(X_\infty, p_\infty, \textbf{v}_{\infty})$ are $(\epsilon_i, R_i)$-related;

		\item (via quasi-conformal maps, \cite[Proposition 3.4.6]{Biringer2021})  There are sequences $K_i \to 1$, $R_i,R_i' \to \infty$, and $K_i$-quasiconformal embeddings 
		\[ \psi_i: B_{X_\infty}(p_\infty, R_i) \to X_i \]
		such that $B_{X_i}(p_{i},R_i') \subset \psi_i(B_{X_\infty}(p_\infty,R_i))$, $\psi_i$ is conformal in a $\delta$-neighborhood of $p_\infty$ for some small $\delta > 0$ with  $\psi_i^{-1}(p_i) \to p_{\infty}$ and $d\psi_i^{-1}(\textbf{v}_i) \to \textbf{v}_{\infty}$;
		\item (via the smooth topology, \cite[Proposition 3.4.3]{Biringer2021}) There are sequences $R_i \to \infty$ and smooth orientation-preserving embeddings
		\[ \psi_i: B_{X_\infty}(p_{\infty},R_i) \to X_i \]
		such that $\psi_i^{-1}(p_i) \to p_\infty$,  $df_i^{-1}(\textbf{v}_i) \to \textbf{v}_\infty$, and $\psi_i^{\ast}(g_{X_i}) \to g_{X_{\infty}}$ in the $C^{\infty}$ topology, where $g_{X_i}$ and $g_{X_{\infty}}$ are the hyperbolic metrics on $X_i$ and $X_{\infty}$, respectively;

	\end{enumerate}
	\end{proposition}

	\section{Some paths in $\framedH2$} \label{paths}

In this section, we produce continuous paths in $\framedH2$ by  (1) moving the base vector, (2) changing the lengths and twists of a collection of curves in the Fenchel-Nielsen coordinates, and (3) shrinking simple closed geodesics to cusps. While the statements may seem rather intuitive, we will need to be  careful about where to assign base vectors.

Unless stated otherwise, a hyperbolic surface is assumed to be connected, oriented, and complete without boundary. We fix once and for all a base vector $(p_{base},\textbf{v}_{base})\in T^1 \bH^2$. 
	
	\subsection{Moving the base vector}
	
	\begin{lemma}\label{move-basepoint}
	Let $X$ be a hyperbolic surface. Then, the map $T^1X \to \framedH2$ defined by $(p,\textbf{v}) \mapsto [(X, p,\textbf{v})]$  is continuous.
	\end{lemma}
	\begin{proof}
	Let $(p_n, \textbf{v}_n) \to (p, \textbf{v})$ be a convergent sequence in $T^1X$. Let $\Gamma < \PSL_2(\R)$ be the subgroup such that $\bquotient{\Gamma}{\bH^2} \approx X$ and $(p_{base}, \textbf{v}_{base}) \in T^1\bH^2$ projects to $(p,\textbf{v}) \in T^1X$. Choose a fundamental domain $D \subset T^1\bH^2$ for the action of $\Gamma$ such that $(p_{base}, \textbf{v}_{base})\in int(D)$. For each $n \in \mathbb{N}$, let $(\widetilde{p}_n,\widetilde{\textbf{v}}_n) \in D$ be the lift of $(p_n, \textbf{v}_n)$. Let $g_n \in \PSL_2(\R)$ be the element taking  $(\widetilde{p}_n,\widetilde{\textbf{v}}_n)$ to $(p_{base},\textbf{v}_{base})$ and set $\Gamma_n = g_n\Gamma g_n^{-1}$. Then, the isometry $T^1\bH^2 \to T^1\bH^2$ given by $(q,\textbf{w}) \mapsto (g_n(q),dg_n(\textbf{w}))$ descends to $T^1(\bquotient{\Gamma}{\bH^2}) \to T^1(\bquotient{\Gamma_n}{\bH^2})$, which yields the equivalence
	\[ (X, p_n, \textbf{v}_n) \sim (\bquotient{\Gamma_n}{\bH^2}, \pi_{\Gamma_n}(p_{base}), d\pi_{\Gamma_n}(\textbf{v}_{base}) ) \]
	in $\framedH2$. Since $g_n\to 1$ and conjugation is continuous, it follows that the subgroups $\Gamma_n$ Chabauty converge to $\Gamma$. This is equivalent to the convergence $(X,p_n,\textbf{v}_n) \to (X,p,\textbf{v})$ in $\framedH2$. 
	\end{proof}

	\subsection{Modifying lengths and twists} \label{modifying-lengths}

Let $S$ be a geodesically complete hyperbolic surface with a non-abelian fundamental group that is not a sphere with three punctures. Fix a geodesic pants decomposition $\mathcal{P}$ as in Theorem \ref{geo-pants}. For each pant component in $S \setminus \mathcal{P}$, the shortest geodesic arcs joining the boundary components are called \textit{seams}. The \textit{Fenchel-Nielsen coordinates} of $S$ measure the length and the displacement of the endpoints of the seams for each curve in $\mathcal{P}$, recorded respectively by $(\ell_S, \tau_S)$, where $\ell_S \in \R_{> 0}^{\mathcal{P}}$ and $\tau_s \in  \R^{\mathcal{P}}$. (Notation: $A^B = \{\text{functions } f: B \to A\}$ for sets $A,B$.) These are defined the same way as in the finite-type case, see \cite{FarbMargalit} for reference, with the usual convention that boundary components of the convex core and curves representing cusps are not assigned twists. In this section, we will show that varying these coordinates of a finite subset of curves in $\mathcal{P}$ is compatible with the topology of $\framedH2$. 

\begin{definition}
Let $S$ be a hyperbolic surface as above. Fix a finite subset $\mathcal{Q}\subset \mathcal{P}$ of simple closed geodesics.  Define
\[ \mathcal{T}(S, \mathcal{P}, \mathcal{Q}) := \left\{ (X, h)  \, \middle\vert 
\begin{small}
	\begin{array}{c} X \text{ is a complete hyperbolic surface, } h:S \to X \text{ is an orientation-preserving} \\ \text{diffeomorphism,}  (\ell_X,\tau_X) \in (\R_{>0}\times \R)^{\mathcal{Q}} \text{ on } \mathcal{Q},  \text{ and }(\ell_X,\tau_X) \equiv (\ell_{S},\tau_{S}) \text{ on } \mathcal{P}\setminus\mathcal{Q} \end{array}
\end{small}
\right\}/\sim. \]
Here, the pair $(X,h)$ is called a \textit{marked hyperbolic surface} with a \textit{marking} $h$ and $(X,h) \sim (X',h')$ if there is an isometry $I: X \to X'$ homotopic to $h' \circ h^{-1}$. The Fenchel-Nielsen coordinates $(\ell_X, \tau_X)$ for $(X,h)$ are calculated using the pants and seams by pushing forward those via $h$.  We also define the following subspace of $\framedH2$:
\[ \mathcal{H}^2(S,\mathcal{P},\mathcal{Q}) :=  \left\{[(X,p,\textbf{v})]\in \framedH2 \, \middle\vert 
\begin{small}
	\begin{array}{c} [(X,h)] \in  \mathcal{T}(S, \mathcal{P}, \mathcal{Q})  \text{ is marked with a marking} \\    h: S \to X \text{ and }(p,\textbf{v}) \in T^1X \end{array}
\end{small}
\right\}. \]
\end{definition}

Since $\mathcal{Q}$ is finite, we may topologize $\mathcal{T}(S,\mathcal{P}, \mathcal{Q})$ by the distance 
\[ d_{FN}((X_1,h_1), (X_2,h_2)) =  \max_{\gamma \in \mathcal{Q}} \left\{\max \left(  \left| \log \frac{\ell_{X_1}(\gamma)}{\ell_{X_2}(\gamma)} \right|, |\ell_{X_1}(\gamma)\tau_{X_1}(\gamma) -\ell_{X_2}(\gamma)\tau_{X_2}(\gamma)  | \right)  \right\}\]
as defined in \cite{Alessandrini2012}. Using standard arguments similar to the finite-type surface case in e.g. \cite{FarbMargalit}, we have that 
\begin{equation}
	\begin{split}
		\Psi:  \mathcal{T}(S,\mathcal{P}, \mathcal{Q}) & \to (\R_{>0}\times \R)^{\mathcal{Q}}  \\
		[(X, h)] &\mapsto \left( \gamma \in \mathcal{Q} \mapsto (\ell_X(\gamma), \tau_X(\gamma)) \right)
	\end{split}
\end{equation}
is a homeomorphism. Indeed, for a pair of functions $\ell: \mathcal{Q} \to \R_{> 0}$ and $\tau: \mathcal{Q} \to \R$, we may choose 
\[ \Psi^{-1}(\ell,\tau) = [(X_{(\ell,\tau)}, h_{(\ell,\tau)})] \in \mathcal{T}(S,\mathcal{P}, \mathcal{Q}) \]
where $h_{(\ell,\tau)}: S \to X_{(\ell,\tau)}$ satisfies
\begin{enumerate}[label = \arabic*), itemsep=0.5ex]
	\item  $h_{(\ell,\tau)}(\gamma)$ is a simple closed geodesic in $X_{(\ell,\tau)}$ for every $\gamma \in \mathcal{P}$, and
	\item after pushing forward $\mathcal{P}$ and the seams, the Fenchel-Nielsen coordinates of $X_{(\ell,\tau)}$ are given by
	\[
	(\ell_{X_{(\ell,\tau)}}(\gamma), \tau_{X_{(\ell,\tau)}}(\gamma)) = 
	\begin{cases}
		\begin{array}{ll}
			(\ell(\gamma), \tau(\gamma)) &\text{if } \gamma \in \mathcal{Q}, \\ 
			(\ell_S(\gamma), \tau_S(\gamma)) &\text{if } \gamma \in \mathcal{P}\!\setminus\!\mathcal{Q}.
		\end{array}
	\end{cases}
	\]
\end{enumerate}
We caution that the two conditions above only define a homeomorphism between the convex cores $h_{(\ell,\tau)}: \CC(S) \to \CC(X_{(\ell,\tau)})$, but one can continuously extend it to the funnel and half-plane components of $S \setminus \CC(S)$ to get $h_{(\ell,\tau)}: S \to X_{(\ell,\tau)}$, see details in the proof of Lemma 5.6.4 in \cite{Biringer2021}.

\begin{proposition}\label{FN-change}
Let $S$ be a hyperbolic surface as above with a pants decomposition $\mathcal{P}$ and a finite subset $\mathcal{Q} \subset \mathcal{P}$. 
Then, there is a continuous surjection
\[ \Pi: \mathcal{T}(S,\mathcal{P}, \mathcal{Q}) \times T^1S \to  \mathcal{H}^2(S,\mathcal{P},\mathcal{Q}).   \]
In particular, $\mathcal{H}^2(S,\mathcal{P},\mathcal{Q}) $ is path-connected.
\end{proposition}

\begin{proof}
Let $[(X, h)] \in \mathcal{T}(S, \mathcal{P},\mathcal{Q})$. Since $\mathcal{Q}$ is finite, we may choose an orientation-preserving diffeomorphism $h: S \to X$ to be a quasi-conformal map in its homotopy class. The rest of the argument here essentially follows the proof of Proposition 4.1.1 in \cite{Biringer2021}.  We have that the equivariant quasi-conformal lift $\tilde h: \bH^2 \to \bH^2$ of $h$ extends continuously to a quasi-symmetric map $\partial \tilde h: \partial \bH^2 \to \partial\bH^2$ between the boundaries at infinity, see \cite[Theorem 1]{Douady1986}. Moreover, the Douady-Earle extension of $\partial \tilde h$ 
\[ \tilde F_{(\ell,\tau, h)}:  \bH^2 \to \bH^2, \] 
where $(\ell,\tau)$ denotes the Fenchel-Nielsen coordinates of $(X,h)$, is a quasi-conformal map, and it is equivariant with respect to deck transformations. So, it descends to a canonical quasi-conformal map
\[ F_{(\ell,\tau,h)}: S \to X \]
which depends only on the homotopy class of $h$ and continuously on $(\ell,\tau)$.

Define	$\Pi: \mathcal{T}(S,\mathcal{P}, \mathcal{Q}) \times T^1S \to  \mathcal{H}^2(S,\mathcal{P},\mathcal{Q}) $ by
\begin{equation}\label{Douady}
	\Pi([(X,h)],(p,\textbf{v})) = [(X, F_{(\ell,\tau,h)}(p), dF_{(\ell,\tau,h)}(\textbf{v}))].
\end{equation}
If $ (X, h) \sim (X', h')$ in $\mathcal{T}(S,\mathcal{P},\mathcal{Q})$, then there is an isometry $I: X \to X'$ homotopic to $h'\circ h^{-1}$ and so by properties of the Douady-Earle extension
\[ I\circ F_{(\ell,\tau,h)} = F_{(\ell,\tau,I \circ h)}  = F_{(\ell, \tau,h')}. \] Thus, $[(X,  F_{(\ell,\tau,h)}(p), dF_{(\ell,\tau,h)}(\textbf{v}))] = [(X', F_{(\ell,\tau,h')}(p), dF_{(\ell,\tau,h')}(\textbf{v}))]$, and $\Pi$ is well-defined.
Continuity of $\Pi$ follows from continuity of  $\tilde F_{(\ell,\tau,h)}$ with respect to the relevant parameters \cite[Proposition 2]{Douady1986}. It is clear that $\Pi$ is surjective. Finally, since $\mathcal{T}(S, \mathcal{P},\mathcal{Q})$ and $T^1S$ are both path-connected, then so is $\framedH2(S,  \mathcal{P},\mathcal{Q})$.
\end{proof}

	\subsection{Pinching a finite collection of disjoint simple closed geodesics}\label{pinching-simple}	
	We start with a geodesically complete hyperbolic surface $S$ such that $\pi_1(S)$ is non-abelian.	Fix a geodesic pants decomposition $\mathcal{P}$ and a collection of seams. These determine the Fenchel-Nielsen coordinates for $S$. We will choose any finite subset $\mathcal{Q} \subset \mathcal{P}$ and show that shrinking the curves in $\mathcal{Q}$ to cusps can produce a continuous path in $\framedH2$, provided that the base vectors are suitably placed. In fact, we will derive this fact as a consequence of a certain quasiconformal embedding of a geodesic pair of pants into one with cusps explicitly constructed by Buser-Makover-M{\"u}tzel-Silhol in \cite{Buser2014}, which we recall and extend in Appendix \ref{appendix}.
		
	Given a family of functions $\ell_t: \mathcal{Q} \to \R_{>0}$ for $t \in [0,1)$, we construct a family of marked hyperbolic surface $\{(X_t,h_t)\}$, each diffeomorphic to $S$, by changing the length coordinate of $\gamma \in \mathcal{Q}$ such that  $\ell_{X_t}(h_t(\gamma)) = \ell_t(\gamma)$ and holding the other length and twist parameters constant. The construction of a marked hyperbolic surface from prescribed Fenchel-Nielsen coordinates is standard, see \cite[Chapter 6]{Buser1992}. We also assume that all markings take simple closed geodesics in $\mathcal{P}$ to simple closed geodesics.
	
 	Choosing a component $S' \subset S \setminus \mathcal{Q}$, we let $X'$ be the hyperbolic surface resulting from replacing the pairs of pants (or a funnel) in $S' \setminus \mathcal{P}$ incident to $\gamma \in \mathcal{Q}$ by ones in which $\gamma$ is a cusp (or a one-cusp cylinder, respectively) and keeping the other lengths and twists the same. This comes with a natural marking $h': int(S') \to X'$. 
	
	\begin{lemma}\label{pinch}
 With the same setup as above, suppose that $\ell_t: \mathcal{Q} \to \R_{>0}$ is a family of functions continuous with respect to  $t\in [0,1)$  and $\ell_t(\gamma) \to 0$ as $t \to 1$ for every $\gamma \in \mathcal{Q}$. Then, for any choice $(p',\textbf{v}') \in T^1X'$, there are choices of base vectors $(p_t, \textbf{v}_t) \in T^1X_t$ such that $(X_t,p_t,\textbf{v}_t)$ is a continuous path in $\framedH2$ and $(X_t,p_t,\textbf{v}_t) \to (X', p',\textbf{v}')$ as $t\to 1$.
	\end{lemma}
	\begin{proof}
	We first produce some continuous convergent path $(X_t, q_t, \textbf{w}_t) \to (X', q', \textbf{w}')$ as $t \to 1$ for some choices of base vectors. For simplicity, we assume that $\mathcal{Q}$ is just a single simple closed geodesic $\gamma$. We apply Proposition \ref{FN-change} to continuously vary the length of $\gamma$ and the base vectors until $\ell_{X_t}(\gamma) < 1/2$. Suppose for now that $\gamma$ is separating and $X'$ is not a one-cusp cylinder. Consider the pair of pants $P$ in the decomposition of $S$ that is adjacent to the cusp corresponding to $\gamma$ in $X'$. Let $\alpha$ be a boundary component or a cusp of $P$ which is not $\gamma$. If $q^0_t \in X_t$ is the point on $h_t(\alpha)$ that is the endpoint of a common perpendicular between $h_t(\alpha)$ and $h_t(\gamma)$, we choose $(q_t, \textbf{w}_t) \in T^1(X_t)$ to be the unit vector tangent to the geodesic ray starting perpendicularly to $h_t(\alpha)$ at  $q_t^0$ at distance $\eta/2$ from $q_t^0$, where $\eta >0 $ is the constant specified in Theorem \ref{thm:app}(3b) in the Appendix. 
	
	To see continuity along the path $(X_t, q_t, \textbf{w}_t)$, we will build a $(\epsilon,R)$-relation between $(X_s,q_s,\textbf{w}_s)$ and $(X_t, q_t, \textbf{w}_t)$ when $|s-t|$ is very small. For  $t \in [0,1)$, cutting along the seams of $h_t(P)$, we have two isometric right-angled hexagons and denote one such hexagon by $H_t$. Let $\alpha_t$ be the side of $H_t$ that is half of $h_t(\alpha)$, $\gamma_t$ half of $h_t(\gamma)$, and $c_t$ the side of $H_t$ corresponding to the seam of $h_t(P)$ that is not incident to $h_t(\gamma)$. 
	
	 Now, for $0<s<t$, we align two hexagons $H_s$ and $H_t$ in the common Fermi coordinate with  $c_s$ and $c_t$ in the horizontal line so that $\alpha_s$ coincides with $\alpha_t$ as illustrated below. Parametrize the horizontal line by unit speed starting at the endpoint of $\alpha_s$ on $c_s$.  For a point $a \in H_s \cup H_t$, let $R(a)$ be the point in $H_s \cap H_t$ closest to $a$. Define $\mathfrak{R}_{s,t} := \{ (x,R(x)) : x \in H_s   \} \cup  \{ (R(y),y) : y \in H_t   \}$. 
	
	\begin{center}
	\begin{figure}[h]
	\includegraphics[scale=0.7]{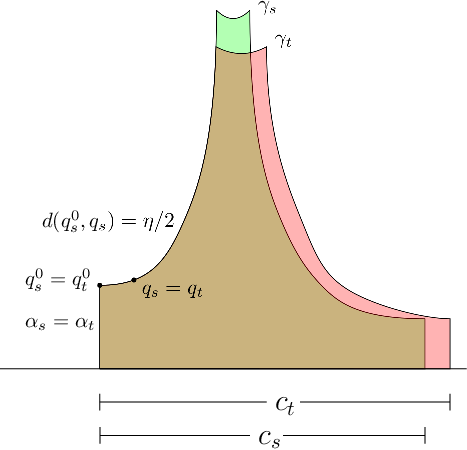}\label{fig: rel}
	\caption{Two hexagons $H_s$ (green) and $H_t$ (red) placed on a common Fermi coordinate with respect to $c_s$ (parametrized with unit speed) so that $\alpha_s = \alpha_t$. The intersection of $H_s$ and $H_t$ is shaded in brown. In the Fermi coordiates with respect to $c_s$, the ordered pair $(\nu,\rho)$ represents the point perpendicular to $c_s(\nu)$ at  distance $\rho$.}
	\end{figure}
	\end{center}
	\vspace{-2em}
	
	By doubling $\mathfrak{R}_{s,t}$ along the boundary of $H_s$ and $H_t$, we define an $(\epsilon_{s,t},R_{s,t})$-relation between $(h_s(P),q_s,\textbf{w}_s)$ and $(h_t(P),q_t,\textbf{w}_t)$ where $\epsilon_{s,t} \to 0$ and $R_{s,t}$ can be chosen to approach infinity as $|s-t| \to 0$. Extending the relation to the rest of $X_s$ and $X_t$ by isometry, we deduce that our path $(X_t, q_t, \textbf{w}_t)$ is continuous. 
	
	Now, we proceed to show that $(X_t, q_t, \textbf{w}_t)$ has a limit when $t \to 1$. Let $\text{Tr}_{t}(X')$ be the subsurface of $X'$ obtained by removing the open horoball neigborhood with horocyclic boundary of length $\frac{2\ell_t(\gamma)}{\pi}$ from the cusp corresponding to $\gamma$. Theorem \ref{thm:app} guarantees a quasiconformal homeomorphism
	 $\phi_t: \text{Tr}(h'(P)) \to h_t(P)$ and we extend it to a quasiconformal homeomorphism between the rest of the surfaces
	\[ \phi_t: \text{Tr}_t(X') \to X_t\]
	 by isometry. Then, $\phi_t$ still has dilatation $\leq 1+ 2\ell_t(\gamma)^2$ and, by our choice of $\eta$, the map $\phi_t$ is conformal on the ball $B(q_t, \eta/2) \subset C_{\eta}(h'(\gamma))$. Setting $q' = \lim\limits_{t\to1} \phi_t(q_t)$ and $\textbf{w}' = \lim\limits_{t\to1} d\phi_t(\textbf{w}_t)$, we have $(X_t, q_t,\textbf{w}_t) \to (X', q',\textbf{w}')$ as $t\to 1$ by Criterion (3) of Proposition \ref{conv-crit}. 
	
	If $X'$ is a one-cusp cylinder, let $F_t \subset X_t$ be the half-infinite funnel bounded by $h_t(\gamma)$. Then, by Corollary \ref{thm:app2}, there is a quasiconformal  homeomorphism
	\[ F_t: \text{Tr}_t(X') \to F_t \]
	with dilatation $\leq 1+ \ell_t(\gamma)^2$ that acts conformally outside the horocyclic neighborhood with boundary length 1. We then choose our base vector in this region and the remaining of the proof works the same way.
	
	If the curve $\gamma \in \mathcal{Q}$ is not separating, we modify the definition of $\text{Tr}_t(X')$ by instead removing from $X'$ the open horoball neigborhood with horocyclic boundary of length $\frac{2\ell_t(\gamma)}{\pi} +\epsilon_t$ for some sequence of small $\epsilon_t > 0$ such that $\epsilon_t\to 0$. Then, by restricting the domain of the map in Theorem \ref{thm:app}, there is a quasiconformal embedding $\phi_t: \text{Tr}_t(X') \to X_t$ and the remainder of the proof works as the separating case.

Finally, let 
\begin{equation*}
		\Gamma = \Phi^{-1}(X', p',\textbf{v}'), \quad
		\Gamma'_t = \Phi^{-1}(X_t, q_t, \textbf{w}_t), \quad
		\Gamma'' = \Phi^{-1}(X', q', \textbf{w}')
\end{equation*}
be the subgroups in $\PSL_2(\R)$ corresponding to the indicated vectored surfaces, where $\Phi$ is defined as in the map (\ref{chabauty-corr}) in Section \ref{sec: chab}. Let $g \in \PSL_2(\R)$ be such that $g\Gamma'' g^{-1} = \Gamma$. The convergence $\Gamma'_t \to \Gamma''$ implies that $g \Gamma'_t g^{-1} \to g\Gamma'' g^{-1} = \Gamma$ as $t \to 1$. Thus, it follows that
\[ (X_t, p_t, \textbf{v}_t):= \Phi(g\Gamma'_t g^{-1} ) \to \Phi(\Gamma) = (X', p', \textbf{v}') \]
as desired. 
	\end{proof}

\section{Theorem \ref{globally connected}, its variants, and some remarks} \label{proof}
We first note that any choice of base vector of $\bH^2$ is equivalent in $\framedH2$ since the isometry group $\PSL_2(\R)$ acts by isometry transitively on $T^1\bH^2$. From now on, let us fix once and for all a base vector $(p_{base},\textbf{v}_{base}) \in T^1\bH^2$. 

Lemma \ref{converge-H2} below is a characterization of sequences converging to $(\bH^2, p_{base}, \textbf{v}_{base})$ in terms of the injectivity radius. Recall that the \textit{injectivity radius} of a point $p$ in a hyperbolic surface $X$, denoted $\inj_X(p)$, is the supremum of the radius $r$ so that the open hyperbolic disk $B_X(p,r)$ isometrically embeds in $X$. 

\begin{lemma}\label{converge-H2}
A path $(X_t,p_t,\textbf{v}_t)$ of vectored surfaces in $\framedH2$, where $t \in [0,1)$, converges to $(\bH^2, p_{base}, \textbf{v}_{base})$ as $t \to 1$ if and only if $\inj_{X_t}(p_t) \to \infty$ as $t \to 1$.
\end{lemma}
\begin{proof}
Let $R_t = \inj_{X_t}(p_t)$. First, suppose that $(X_t, p_t,\textbf{v}_t) \to (\bH^2, p_{base}, \textbf{v}_{base})$ but that there is $R>0$ such that $R_{t_j} < R$ for some subsequence ${t_j}$. Thus, $B_{X_{t_j}}(p_{t_j},R)$ contains a homotopically nontrivial loop, and so there can be no diffeomorphism between any metric disk in $\bH^2$ and $B_{X_{t_j}}(p_{t_j},R)$, which is a contradiction. Conversely, if $R_t \to \infty$, then there is clearly an isometry between $B_{X_t}(p_t,R_t)$ and $B_{\bH^2}(p_{base}, R_t)$ taking $\textbf{v}_t$ to $\textbf{v}_0$, and the result follows.
\end{proof}

We are now ready to prove Theorem \ref{globally connected}.

\begin{proof}[Proof of Theorem \ref{globally connected}]
We introduce a temporary notation to make the proof more concise. For two hyperbolic surfaces $X$ and $Y$, we say $X\sim_C Y$ if there exist choices of base vectors $(p,\textbf{v}) \in T^1X$, $(q,\textbf{w}) \in T^1_qY$ for which $(X,p,\textbf{v})$ and $(Y,q,\textbf{w})$ belong to the same path component of $\framedH2$. Clearly, $\sim_C$ is reflexive and symmetric. As we are allowed to move base vectors (Lemma \ref{move-basepoint}), $\sim_C$ is also transitive. Hence, it is an equivalence relation. 

Our strategy is to show that $X \sim_C \bH^2$ for any vectored hyperbolic surface $(X,p, \textbf{v}) \in \framedH2$. We are done if $X$ is simply connected, in which case $X$ is isometric to $\bH^2$. If $X$ is topologically a cylinder, then $X$ is isometric to a quotient of $\bH^2$ by some cyclic subgroup generated by either a single parabolic-type or hyperbolic-type isometry in $\PSL_2(\R)$. In either case, we can move the basepoint along a path which asymptotically increases the injectivity radius without bound and thus $X \sim_C \bH^2$ by Lemma \ref{converge-H2}.

Now, we can suppose that $X$ admits a (possibly empty) geodesic pants decomposition $\mathcal{P}$ as in Theorem \ref{geo-pants} such that there is a component $Y$ of $X \setminus \mathcal{P}$ whose closure is a generalized pair of pants. Let $Y_0$ be the pair of pants with three cusps. Choose some base vector $(p_0,\textbf{v}_0) \in T^1Y_0$. Then, applying Lemma \ref{pinch} as we shrink the boundary curves of $Y$ to zero, we produce a continuous path from $(X, p',\textbf{v}')$ to $(Y_0, p_0,\textbf{v}_0)$ for some choice $(p',\textbf{v}') \in T^1X$ given by the lemma. Thus, $X \sim_C Y_0$.

It then remains to show that $Y_0 \sim_C \bH^2$. To do this, consider the geodesic pair of pants $Y_1$ all of whose boundary lengths are 1 and the hyperbolic surface $\widetilde{Y_1}$ obtained by gluing a funnel to each boundary component of $Y_1$. On one hand, we can take a base vector out a funnel in $Y_1$, and so $Y_1 \sim_C \bH^2$. On the other hand, since $Y_1 \subset \widetilde{Y}_1$, we return to the previous case, shrink the boundary curves of $Y_1$ and show that $Y_1 \sim_C Y_0$. Thus, $Y_0 \sim_C \bH^2$ as desired.
\end{proof}

Finally, we record some path-connected subspaces of $\framedH2$, which might be of some interest.

\begin{proposition}
The following subspaces of $\framedH2$ are path-connected:
\begin{enumerate}
\item For a fixed hyperbolic surface $S$ of finite topological type without boundary, the subspace of vectored hyperbolic surfaces diffeomorphic to $S$:
 \[ \framedH2_{\approx}(S) := \{ (X, p, \textbf{v}) \in \framedH2: X \approx S  \}.\]
 \item If $S$ is a hyperbolic surface of finite type, possibly with geodesic boundary, the subspace of vectored hyperbolic surfaces into which the convex core of $S$ embeds $\pi_1$-injectively: 
 \[ \framedH2_{\supset}(S) := \{ (X,p,\textbf{v}) \in \framedH2: \text{there exists a $\pi_1$-injective smooth embedding } \phi: \CC(S) \hookrightarrow X\}\]
 \item The subspace of vectored hyperbolic surfaces whose area is bounded above by some fixed constant.
 \item The subspace of vectored hyperbolic surfaces with finite area.
\end{enumerate}
\end{proposition}

\begin{proof}
For (1), fix a pants decomposition $\mathcal{P}$ of $S$ and this follows from Proposition \ref{FN-change} by taking $\mathcal{Q} = \mathcal{P}$. 

For (2), let $S'$ be the hyperbolic surface obtained by shrinking every boundary component (if any) of $\CC(S)$ to a cusp. Note that $\framedH2_{\supset}(S)$ contains a vectored $S'$. Now, given any $(X, p,\textbf{v}) \in \framedH2_{\supset}(S)$ with a $\pi_1$-injective embedding $\phi: \CC(S) \xhookrightarrow{} X$, we let $X_S$ be the geodesic subsurface of $X$ (possibly with geodesic boundary) such that $int(X_S)$ is homotopic to $\phi(int(\CC(S)))$. This map also induces the Fenchel-Nielsen coordinates for $X_S$ by considering the simple closed geodesics (or cusps) homotopic to the image under $\phi$ of pants curves in $S$. Applying Lemma \ref{FN-change} to change these Fenchel-Nielsen coordinates for $X_S$ to match those of $S$ and Lemma \ref{pinch} to shrink all boundary curves of $\CC(X_S))$ to cusps, we have a path from $(X, p,\textbf{v}) $ to a vectored $S'$.

By pinching the boundary curves of a single pair of pants in a pants decomposition to cusps, we produce a path from a vectored hyperbolic surface with finite area to a vectored thrice punctured sphere (Lemma \ref{pinch}). Since this process does not increase area, items (3) and (4) follow immediately. 
\end{proof}

\begin{remark}
As in \cite{Biringer2021}, we may consider the space of vectored hyperbolic 2-orbifolds or, equivalently, the Chabauty space of non-elementary subgroups of $G = \PSL_2(\R)$. However, this space is not path-connected. For example, if $S$ is a sphere with three cone points of some fixed order, then
\[  \Sub_S(G)  := \{ \Gamma \leq G: \bquotient{\Gamma }{\bH^2} \approx S \}\]
is both open and closed. That it is open follows from compactness of $S$: by the orbifold version of (2) in Proposition \ref{conv-crit}, any vectored 2-orbifold close to, say, $(S, p,\textbf{v})$ admits a diffeomorphism $B_S(p,R) \to S$, where $R > \diam(S)$, so in fact it must be diffeomorphic to $S$. It is closed because the moduli space of $S$ is a point and $T^1S$ is compact. By moving the base vector, we see that  $\Sub_S(G)$ is path-connected. Thus, $\Sub_S(G)$ is its own path component.
\end{remark} 

\begin{remark} \label{rem: higherd}
As mentioned in the introduction, the space $\framedHn$ of (isometry classes of) \textit{framed hyperbolic $n$-manifolds} can also be defined for higher dimensions ($n \geq 3$). Here, a framed manifold means the manifold is equipped with a choice of an orthonormal basis of the tangent space at the basepoint. The convergence criteria as in Proposition \ref{conv-crit} still hold. 

However, $\framedHn$ is not path-connected. Indeed, if $M$ is a closed hyperbolic $n$-manifold, then we claim that
\[ \{ [(M, p ,\textbf{e})] : p \in M, \textbf{e} \text{ is an orthonormal basis of } T^1_p(M) \}\] 
is a path component in $\framedHn$. By moving the baseframe, we have that this subspace is path-connected. Moreover, it is closed by compactness of $T^1M$ and open,  since a framed hyperbolic manifold $(M',p',\textbf{e}')$ close enough to $(M,p,\textbf{e})$ in $\framedHn$ has an embedding $B_M(p,R) \to M$. By taking $R > \diam(M)$, we have that $M'$ must be diffeomorphic to $M$, hence isometric to $M$ by Mostow's rigidity. 
\end{remark}

\appendix

\section{Quasiconformal maps between pairs of pants}\label{appendix}
For the sake of completeness, we record some properties of the quasiconformal map between pairs of pants constructed in \cite{Buser2014} and then use them to define a quasiconformal map between a truncated one-cusp cylinder and a half-infinite funnel. These are used in the proof of Lemma \ref{pinch}. A few terminologies are required for the setting.

 Given $L: \{1,2,3 \} \to \R_{\geq 0}$, let $P_{L}$ be the generalized geodesic pair of pants with boundary geodesics $\gamma_1, \gamma_2, \gamma_3$ such that $\ell_{P_{L}}(\gamma_i) = L(i)$. Furthermore, if $\mathcal{C} \subset \{1,2,3\}$, we let $P_{ \mathcal{C},L}$ be the geodesic generalized pair of pants with boundary components $\gamma'_1, \gamma'_2, \gamma'_3$ such that 
\[\ell_{P_{\mathcal{C},L}}(\gamma'_i) = \begin{cases}
	L(i) & \text{if } i \not \in \mathcal{C} \\
	0 & \text{if } i \in \mathcal{C}.
\end{cases} \]

For a boundary geodesic $\gamma_i$ of $P_{L}$ (and analogously for $\gamma'_i$ of $P_{\mathcal{C},L}$) and $w > 0$ smaller than the Margulis constant, we have the \textit{$w$-collar} 
\[ C_{w}(\gamma_i) := \{ x \in P_{L}: d(x, \gamma) < w  \} \]
and the \textit{$\gamma_i$-equidistant curve} 
\[ \gamma_i^w := \{ x \in P_{L}: d(x, \gamma) = w\}.\]
If $\gamma_i$ represents a cusp of $P_{L}$ (and similarly for $\gamma'_i$ of $P_{\mathcal{C},L}$), we let $C_w(\gamma_i)$ be the open horoball neighborhood of $\gamma$ with horocyclic boundary of length $w$. In this case, let $\gamma_i^w$ be the horocycle  of $\gamma_i$ length $w$. We will also denote by $C_{(w_1, w_2)}(\gamma_i)$, where $0< w_1 < w_2$, the open cylindrical subset of the horoball neighborhood of $\gamma_i$ bounded by horocycles of length $w_1$ and $w_2$.

Each $\gamma_i$  (and analogously for $\gamma'_i$) is parametrized on $[0,1]$ by constant speed, and it is oriented so that $P_{L}$ is to the left and that $\gamma_i(0)$ is the endpoint of the seam joining $\gamma_i$ and $\gamma_{i+1}$ (the indices are read mod 3). This parametrization extends to the curves $\gamma_i^w$ as follows: $\gamma_i^w(t)$ is the point on $\gamma_i^w$ that intersects the geodesic ray in $P_L$ perpendicular to $\gamma_i$ starting at $\gamma_i(t)$ for $t \in [0,1]$. We call this the \textit{standard parametrization}.  

Finally, we define the \textit{$\mathcal{C}$-truncated pair of pants} to be  
\[ \text{Tr}(P_{\mathcal{C},L}) := P_{\mathcal{C},L} \setminus \bigcup_{i \in \mathcal{C}} C_{\frac{2L(i)}{\pi}}(\gamma'_i). \]

We now recall the following result by Buser-Makover-M{\"u}tzel-Silhol \cite{Buser2014}, rephrased slightly for our purpose (see Proposition \ref{conv-crit}).

\begin{theorem}[\cite{Buser2014}]\label{thm:app}
Given $L: \{1,2,3 \} \to \R_{\geq 0}$ and $\mathcal{C} \subset \{1,2,3\}$, if $\epsilon:= \max\limits_{i \in \mathcal{C}} L(i) < 1/2$, there is a quasiconformal homeomorphism
\begin{equation*}\label{eq: qi-map} \phi_L:  \text{Tr} (P_{\mathcal{C},L}) \to P_{L}
\end{equation*}
which has the following properties
\begin{enumerate}
	\item $\phi_L$ has the dilatation $q_{\phi_L} \leq (1 + 2\epsilon^2)^{|\mathcal{C}|}$
	\item  if $i \in \mathcal{C}$ has  $L(i) >0$, setting $w_i = 2L(i)/\pi$ and $\delta_i = \log\left( 2/L(i)\right) $, then 	$\phi_L(C_{(w_i,1)}(\gamma'_i)) = C_{\delta_i}(\gamma_i)$.  Moreover, $\phi_L$ respects the standard parametrization of the boundary of $C_{(w_i, 1)}(\gamma'_i)$ in the sense that $ \phi_L((\gamma'_{i})^{w_i}(t)) = \gamma_i(t) $ and $ \phi_L((\gamma'_{i})^{1}(t)) = (\gamma_i)^{\delta_i}(t)$ for all $t \in [0,1]$
	\item if $i \not\in \mathcal{C}$,  the map $\phi_L$ is an isometry on $C_{\eta_i}(\gamma'_i)$, where $\eta_i > 0$ is defined as follows
	\begin{enumerate}
	\item if $\gamma'_i$ is a cusp, set $\eta_i =1$
	\item if $\gamma'_i$ is a simple closed geodesic \footnote{This value of $\eta_i$ follows immediately from the proof of the map in Theorem 2.1 of  \cite{Buser2014} and a trigonometric formula of a trirectangle (see the Appendix of \cite{Buser1992}) by realizing that a collar of this width lies in the region $ABDH$ in Figure 2 of their paper.}, set $\eta_i = \tanh^{-1}\left(\frac{\tanh(\omega_i)}{\cosh(L_i/2)}\right)$ where $\omega_i = \sinh^{-1}\left(\frac{1}{\sinh(L_i/2)}\right)$ (for simplicity, $\eta_i = \log(2/L_i)$ works for $L_i<2$). 
	\end{enumerate}
\end{enumerate}
\end{theorem}   

\begin{figure}[h]
	\includegraphics{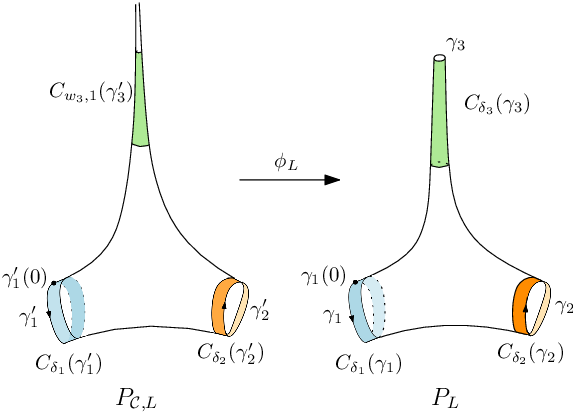}
	\caption{An illustration of the quasiconformal map $\phi_L:  \text{Tr} (P_{\mathcal{C},L}) \to P_{L}$ constructed in \cite{Buser2014}. Here, $\mathcal{C} = \{3\}$. For $i=1,2$, the $\eta_i$-collar neighborhood of $\gamma'_i$ is mapped isometrically to the $\eta_i$-collar neighborhood of $\gamma_i$, colored in blue and orange. The green collar neigborhood of $P_L$ is the image under $\phi_L$ of the horocyclic cylinder in $P_{\mathcal{C},L}$, shaded also in green.}
\end{figure}

Now, for $\ell>0$, let $F_{\ell}$ be the half-infinite cylinder with geodesic boundary $\gamma$ of length $\ell$ and let $F_0$ be the one-cusp cylinder with the cusp represented by $\gamma'$. We can similarly define the $w$-collar $C_w(\gamma) \subset F_{\ell}$, $C_w(\gamma') \subset F_0$, and $C_{(w,1)}(\gamma') \subset F_0$ for $0< w<1$. 

\begin{corollary}\label{thm:app2}
	If $0< \ell < 1/2$, setting $\ell^* = 2\ell/\pi$ and letting $\text{Tr}(F_0): = F_0 \setminus C_{\ell^\ast}(\gamma')$, then there is a quasiconformal homeomorphism
	\[ \phi: \text{Tr}(F_0) \to F_{\ell} \]
	that preserves the standard parametrization of the boundary.  Moreover, the dilatation is $q_{\phi}\leq 1+2\ell^2$.
\end{corollary}
\begin{proof}
Theorem \ref{thm:app} above guarantees a quasiconformal homeomorphism 
\[ \phi: \overline{C_{(\ell^\ast,1)}(\gamma')} \to \overline{C_{\log(\ell^\ast)}(\gamma)} \]
with dilatation $\leq 1+2\ell^2$, and $\phi$ respects the standard parametrization of the boundary components.  We extend $\phi$ to a homeomorphism
\[ \tilde{\phi}: \text{Tr}(F_0) \to F_\ell \]
as follows. Give the horocycle $(\gamma')^{1} \subset \text{Tr}(F_0)$ the standard parametrization. For each $t \in [0,1]$, let $u_t$ be the unit-speed geodesic ray in $F_0$ starting perpendicularly at $(\gamma')^1(t)$ and going away from the cusp. For $s \in (0,\infty)$, parametrize the horocycle $(\gamma')^{1+s}$ so that $(\gamma')^{1+s} \cap u_t = \{(\gamma)'^{1+s}(t)\}$.

Then, on  $\text{Tr}(F_0)\setminus C_{(\ell^\ast,1)}(\gamma')$, we set 
\[\phi((\gamma')^{1+s}(t)) := \gamma^{\log(\ell^\ast)+s}(t).\]
This sends the orthogonal foliation of $\text{Tr}(F_0)\setminus C_{(\ell^\ast,1)}(\gamma')$ by the horocycles of $\gamma'$ and the $u_t$ to the orthogonal folation of $F_{\ell}$ by the equidistant curves to $\gamma'$  and the geodesic rays perpendicular to $\gamma'$. Since this extension is conformal and agrees with $\phi$ on the boundary, the result follows.
\end{proof}

\begin{center}
	\begin{figure}[h]
		\includegraphics[scale=0.7]{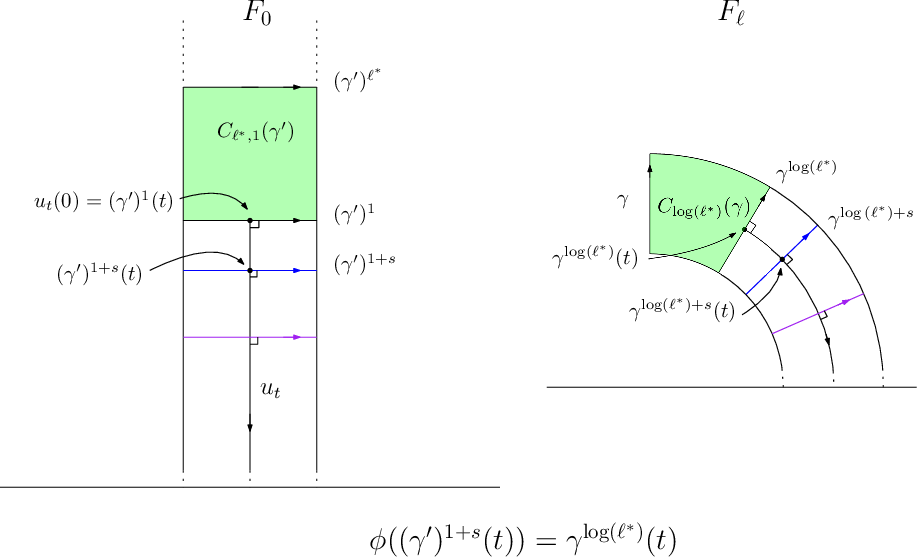}
		\caption{An illustration of the map $\phi: \text{Tr}(F_0) \to F_{\ell} $, unfolded in the upper half plane model, where the function is defined as in \cite{Buser2014} inside the green region and takes the orthogonal foliation by horocycles of $\gamma'$ in $F_0$ to the orthogonal by $\gamma$-equidistant curves in $F_{\ell}$ outside the green region.}
	\end{figure}
\end{center}

 \subsection*{Acknowledgements}
The author is indebted to his Ph.D. advisor Ian Biringer who introduced the topic and shared his insights with kind patience. The author would also like to thank Yi Huang and Guillaume Tahar for conversations, and the referee for useful comments that helped improve earlier drafts of this paper. 

\bibliographystyle{plain}
\small{
\bibliography{reference_arxiv}{}
}

\end{document}